\documentclass[12pt]{amsart}

\usepackage[utf8]{inputenc}
\usepackage[a4paper,nomarginpar]{geometry}

\usepackage[english]{babel}
\usepackage{enumitem}
\usepackage{amsmath,amsfonts,amssymb,amsthm}
\usepackage{amscd}
\usepackage{tikz}
\usepackage[pagewise]{lineno}

\allowdisplaybreaks

\makeatletter
\def\subsection{\@startsection{subsection}{2}%
	\z@{.5\linespacing\@plus.7\linespacing}{.3\linespacing}%
	{\normalfont\bfseries}}
\makeatother

\theoremstyle{theorem}
\newtheorem{theorem}{Theorem}
\newtheorem{lemma}[theorem]{Lemma}

\theoremstyle{definition}

\newtheorem{remark}[theorem]{Remark}

\theoremstyle{remark} \theoremstyle{question} \theoremstyle{example}

\newcommand{\N}{\mathbb{N}}   
\newcommand{\R}{\mathbb{R}}   

\newcommand{\V}{\mathbb{V}}
\newcommand{\W}{\mathbb{W}}

\newcommand{\cM}{\mathcal{M}}

\newcommand{\cP}{\mathcal{P}}

\newcommand{\eps}{\varepsilon}   
\newcommand{\ov} {\overline}     
\newcommand{\wt} {\widetilde}    

\newcommand{\oA}{{\boldsymbol{A}}}
\newcommand{\oU}{{\boldsymbol{U}}}

\newcommand{\upind}{{\overline{\rm I}}}

\begin{document}
	
	
	\title[Measure-theoretic Uniformly Positive Entropy On The Space of Probability Measures]
	{Measure-theoretic Uniformly Positive Entropy\\ On The Space of Probability Measures}
	
	
	
	\author{R\^omulo M. Vermersch}
	\address{Departamento de Matem\'atica, Centro de Ci\^encias F\'isicas e
		      Matem\'aticas, Universidade Federal de Santa Catarina,
		   Florian\'opolis, SC, 88040-900, Brazil.}
	\curraddr{}
	\email{romulo.vermersch@ufsc.br}
	\thanks{}
	
	
	
	\date{}
	
	\dedicatory{}
	
	\maketitle
	\begin{abstract}
		
		For a homeomorphism $T$ on a compact metric space $X$, a $T$-invariant Borel probability measure $\mu$ on $X$ and a measure-theoretic quasifactor $\widetilde{\mu}$ of $\mu$, we study
		the relationship between the local entropy of the system $(X,\mu,T)$ and of its induced system $(\mathcal{M}(X),\widetilde{\mu},\widetilde{T})$, where $\widetilde{T}$ is the homeomorphism induced by $T$ on the space $\mathcal{M}(X)$ of all Borel probability measures
		defined on $X$.
	\end{abstract}


	\section{Introduction}\label{Intro}
	
	By a {\em topological dynamical system} (TDS) we mean a pair $(X,T)$
	consisting of a compact metric space $X$ with metric $d$ and a homeomorphism $T : X \to X$. Such a TDS induces, in a natural way, the TDS $(\cM(X),\wt{T}).$ Here, $\cM(X)$ denotes the space of all Borel probability
	measures on $X$ endowed with the {\em Prohorov metric}
	$$
	d_P(\mu,\nu):= \inf\{\delta > 0 : \mu(A) \leq \nu(A^\delta) + \delta
	\text{ for all } A \in \mathcal{X}\},
	$$
	where $\mathcal{X}$ is the $\sigma$-algebra of all Borel subsets of $X$, and
	$\wt{T} : \cM(X) \to \cM(X)$ is the homeomorphism given by
	$$
	(\wt{T}(\mu))(A):= \mu(T^{-1}(A)) \ \ \ (\mu \in \cM(X), A \in \mathcal{X}).
	$$
	It is well known that  $\cM(X)$ is a compact metric space. We refer the reader to the books
	\cite{PBil99,RDud02} for a study of the space $\cM(X)$.
	
	The investigations between the dynamics of the TDS $(X,T)$ and
	the dynamics of the induced TDS $(\cM(X),\wt{T})$ were initiated by Bauer and Sigmund \cite{WBauKSig75},
	and later were developed by several authors; see
	\cite{NBerUDarRVer22,NBerRVer16,EGlaBWei95,EGlaBWei03,QiaoZhou17,KSig78}, for instance.

	By a {\em measure-theoretic dynamical system} (MDS) we mean a triple $(X,\mu,T)$, where $(X,T)$ is a TDS and $\mu$ is a $T$-invariant Borel probability measure on $X$. Now, recall the notion of a measure-theoretic quasifactor of a MDS $\mathfrak{X}=(X,\mu,T)$ due to Glasner \cite{EGla83} (see also \cite{EGlaBWei95,EGlaBWei03}): A measure-theoretic {\it quasifactor} of $\mathfrak{X}$ is a $\wt{T}$-invariant Borel probability measure $\wt{\mu}$ on $\mathcal{M}(X)$ which satisfies the so-called {\it barycenter equation}:
	\begin{equation*}\label{barycenter} 
		\mu=\displaystyle\int_{\mathcal{M}(X)}\theta d\wt{\mu}(\theta).
	\end{equation*} Equivalently, we say that $\mu$ is the barycenter of $\tilde{\mu}$. The barycenter equation means that, by choosing any compact topology on $X$ compatible with its Borel structure, one has
	\begin{equation*}\displaystyle\int_{X}f(x)d \mu(x)=\int_{\mathcal{M}(X)}\int_{X}f(x)d \theta(x)d \wt{\mu}(\theta)
	\end{equation*}
	for every continuous function $f:X\rightarrow\mathbb{R}$. This definition is independent on the choice of the compact topology compatible with the Borel structure \cite{EGla83}. We denote by $Q(\mu)$ the set of all measure-theoretic quasifactors of $\mu$ and, sometimes, we will say that the induced MDS $\mathfrak{\widetilde{X}}=(\mathcal{M}(X),\widetilde{\mu},\widetilde{T})$ is a measure-theoretic quasifactor of $\mathfrak{X}=(X,\mu,T)$.
	
	In this work we are concerned with the relationship between the local entropy of the measure-theoretic dynamical systems $(X,\mu,T)$ and $(\mathcal{M}(X),\widetilde{\mu},\widetilde{T})$, where $\widetilde{\mu}$ is a quasifactor of $\mu$.
	
	The research on the relationship between the entropy of a MDS and of a measure-theoretic quasifactor of it can be traced back to a deep result due to Glasner and Weiss \cite{EGlaBWei95} which asserts that
	if $(X,\mu,T)$ has zero entropy, then so does $(\mathcal{M}(X),\widetilde{\mu},\widetilde{T})$ for every $\tilde{\mu}\in Q(\mu)$. By using the variational principle, it follows that if $(X,T)$ has zero topological entropy, then so does $(\mathcal{M}(X),\tilde{T})$.
	Qiao and Zhou \cite{QiaoZhou17} obtained such a result for the
	notion of sequence entropy. We point out that our second result in this work is a version of the aforementioned Glasner and Weiss result in the context of local entropy theory \cite{EGlaYe09}. In other direction, Glasner and Weiss \cite{EGlaBWei03} proved that any ergodic system of positive entropy admits \emph{every} ergodic system of positive entropy as a measure-theoretic quasifactor. So, in particular, one sees that the set of measure-theoretic quasifactors of an ergodic system of positive entropy is very large. 
	
	In the present work we study, with the powerful tools of local entropy theory \cite{EGlaYe09,KerrLi09}, the relationship between the entropy of a measure-theoretic dynamical system and the entropy of a measure-theoretic quasifactor of it. More precisely, in our first result we give a characterization of measure-theoretic uniformly positive entropy for a topological dynamical system in terms of the induced system on spaces of probability measures which are concentrated on a finite number of points. In our second result we show that, given any ergodic system, if a measure-theoretic quasifactor of it has measure-theoretic uniformly positive entropy, then so does the given ergodic system. We point out that this second result can be seen as the local version of the corresponding result due to Glasner and Weiss for quasifactors of zero-entropy systems \cite{EGlaBWei95}.
	Recently, Liu and Wei \cite{LiuWei23} studied the relationship between the local entropy of a topological factor of a dynamical system and the local entropy of the topological system induced by that factor on the space of Borel probability measures.

	We now recall the notion of {\em entropy pairs} for a TDS/MDS. By following \cite{Blan92}, we say that a pair of distinct points $x,x^{\prime}\in X$ is an {\em entropy pair} for the TDS $(X,T)$ if, for every open cover $\mathcal{U}=\{U,V\}$ of $X$ with $x\in U^{c}$ and $x^{\prime}\in V^{c}$, one has $h_{top}(T,\mathcal{U})>0$. In \cite{Blan92} Blanchard proved that $(X,T)$ has positive topological entropy if and only if there exists an entropy pair $(x,x^{\prime})$ for $(X,T)$. So, entropy pairs are a tool for localising the topological entropy in the Cartesian product $X\times X$. We denote by $E_{X}$ the set of all entropy pairs for the TDS $(X,T)$. By following \cite{B-H-M-M-R-95}, we say that a pair of distinct points $x,x^{\prime}\in X$ is a {\em $\mu$-entropy pair} for the MDS $(X,\mu,T)$ if, for any measurable partition $\mathcal{P}=\{F,F^{c}\}$ of $X$ such that $F$ contains a neighborhood of $x$ and $F^{c}$ contains a neighborhood of $x^{\prime}$, one has $h_{\mu}(T,\mathcal{P})>0$. We denote by $E_{\mu}$ the set of all $\mu$-entropy pairs for the MDS $(X,\mu,T)$. It was proved in \cite{B-H-M-M-R-95} that $E_{\mu}\subset E_{X}$ for each $T$-invariant Borel probability measure on $X$ and that the reverse inclusion holds when $\mu$ is uniquely ergodic. On the other hand, in \cite{B-G-H-97} it was proved that if $(X,T)$ has positive topological entropy, then there exists a $T$-invariant Borel probability measure $\mu$ on $X$ such that $E_{\mu}=E_{X}$. 
	
	Now we recall the notion of {\it uniformly positive entropy} (UPE). We say that a TDS $(X,T)$ has UPE if every pair of distinct points $x,x^{\prime}\in X$ is in $E_{X}$. This notion was introduced by Blanchard \cite{Blan92} as a candidate for
	an analogue in topological dynamics for the notion of a $K$-process in
	ergodic theory. As it was shown by Blanchard in \cite{Blan92} and \cite{Blan93}, the notion of UPE is more adequate than the notion of {\em topological completely positive entropy} (topological CPE - that is, every topological factor of the TDS has postive entropy) in that respect. Indeed,  in \cite{Blan92} he proved that every non-trivial
	factor of an UPE system has positive topological entropy and in \cite{Blan93}
	he proved that an UPE system is disjoint from every minimal zero entropy
	system, while also in \cite{Blan92} he showed that topological CPE does not implies any degree of topological mixing, not even transitivity. Although an UPE system is topologically weakly mixing but not always
	strongly mixing \cite{Blan93}, Glasner and Weiss \cite{EGlaBWei94} proved
	that UPE is a necessary condition for a TDS $(X,T)$ to have a $T$-invariant
	probability measure $\mu$ of full support whose corresponding measurable
	dynamical system $(X,\mu,T)$ is a $K$-process. Moreover, Huang and Ye \cite{HuangYe} generalized this by proving that if a topological dynamical system admits an invariant $K$-measure with full support, then it has UPE of all orders; so, UPE of all orders is the topological analogue of the $K$-property from Ergodic Theory.
	
	We now introduce the notion of a $\mu$-UPE system. Let $\mathfrak{X}=(X,\mu,T)$ be a MDS. We say that a two-set Borel partition $\mathcal{P}=\{P_{0},P_{1}\}$ of $X$ is a \emph{replete partition} if $\text{int}P_{0}\neq\varnothing$ and $\text{int}P_{1}\neq\varnothing$ \cite{B-H-M-M-R-95}. We say that $\mathfrak{X}$ has $\mu$-UPE if any pair of distinct points $x,x^{\prime}\in X$ is in $E_{\mu}$. Equivalently,  $\mathfrak{X}$ has $\mu$-UPE if $h_{\mu}(T,\mathcal{P})>0$ for every replete partition $\mathcal{P}$. 
	
	In the proofs of our results
	we will use local entropy theory as in the unifying work due to
	Kerr and Li \cite{KerrLi09}.
	
	Finally, we remark that the present work represents the measure-theoretic counterpart of \cite{NBerUDarRVer22} and complements the works \cite{EGlaBWei95} and \cite{EGlaBWei03} in the context of local entropy.


	\section{Preliminaries}
	
	Recall that the Prohorov metric $d_P$ on $\cM(X)$ induces the so-called
	{\em weak*-topology}, that is, the topology whose basic open neighborhoods
	of $\mu \in \cM(X)$ are the sets of the form
	$$
	\V(\mu;f_1,\ldots,f_k;\eps):= \Big\{\nu \in \cM(X) :
	\Big| \int_X f_i \,d\nu - \int_X f_i \,d\mu \Big| < \eps
	\text{ for } i = 1,\ldots,k\Big\},
	$$
	where $k \geq 1$, $f_1,\ldots,f_k : X \to \R$ are continuous functions and
	$\eps > 0$.
	
	In \cite{NBerUDarRVer22} Bernardes, Darji and the author proved the following result that gives another basis for the  weak*-topology on $\mathcal{M}(X)$:

	\begin{lemma}\label{newbasis}
		The sets of the form
		\begin{equation*}\label{Equa1}
			\W(U_1,\ldots,U_k;\eta_1,\ldots,\eta_k):=
			\{\nu \in \cM(X) : \nu(U_i) > \eta_i \text{ for } i = 1,\ldots,k\},
		\end{equation*}
		where $k \geq 1$, $U_1,\ldots,U_k$ are nonempty disjoint open sets in $X$
		and $\eta_1,\ldots,\eta_k$ are positive real numbers with
		$\eta_1+\cdots+\eta_k < 1$, form a basis for the weak$\,^*$-topology on
		$\cM(X)$.
	\end{lemma}
	
	Let us now recall some definitions and notations from entropy theory.
	In what follows, all logarithms are in base $2$. 
	
	Let $(X,\mu,T)$ be a MDS.
	Given finite measurable partitions $\cP_1,\ldots,\cP_n$ of $X$, let
	$$
	\cP_1 \lor \cdots \lor \cP_n := \{P_1 \cap \cdots \cap P_n:
	P_1 \in \cP_1,\ldots,P_n \in \cP_n\}.
	$$
	The {\em entropy} of a finite partition $\cP$ of $X$ is defined by
	$$
	H_{\mu}(\cP):= -\sum_{P\in\cP}\mu(P)\log \mu(P).
	$$ The {\em entropy of $T$ with respect to $\cP$} is defined by
	$$
	h_{\mu}(T,\cP):= \lim_{n \to \infty} \frac{1}{n} H_{\mu}(\cP^{n-1}),
	$$
	where $\cP^{n-1}:= \cP \lor T^{-1}\cP \lor\dots\lor T^{-(n-1)}\cP$,
	and the {\em entropy} of $T$ is given by
	$$
	h_{\mu}(T):= \sup_{\cP} h_{\mu}(T,\cP),
	$$
	where the supremum is taken over all finite measurable partitions $\mathcal{P}$ of $X$.
	
	The notion of  entropy was introduced in Ergodic Theory by Kolmogorov \cite{Kolmogorov}. It plays a major role in dynamics, notably in the Isomorphism Theory \cite{Ornstein75}. 

	We now introduce some elements of Kerr-Li machinery \cite{KerrLi09} that we shall use in the sequel. Let $\mathfrak{X}=(X,\mu,T)$ be a MDS and let $\oA = (A_1 ,\dots ,A_k )$ be a
	tuple of subsets of $X$. For a subset $D$ of $X$, we say that $J\subset\mathbb{N}$ is an
	{\emph independence set for $\oA$ relative to $D$} if for every nonempty
	finite subset $I\subset J$ and every map $\sigma : I\to \{ 1,\dots ,k \}$, we have
	$$D\cap \bigcap_{j\in I} T^{-j} A_{\sigma (j)} \neq\emptyset.$$
For each $\delta > 0$, we denote by $\mathcal{B} (\mu , \delta )$ the collection of all
	Borel subsets $D$ of $X$ such that $\mu (D) \geq 1 - \delta$. For each $m\geqslant 1$ and $\delta>0$, we define
	\begin{align*}
		\varphi({\oA ,\delta},m) &= \min_{D\in \mathcal{B} (\mu , \delta )}
		\max \big\{ |\{1,\dots,m\}\cap J| : J\text{ is an independence set for }\oA
		\text{ relative to } D \big\}.
	\end{align*}
Now, put $$\upind_\mu (\oA , \delta ):=\displaystyle \limsup_{m\to\infty} \frac{\varphi(\oA ,\delta,m)}{m}.$$ Finally, let us define the {\em upper $\mu$-independence density} of $\oA$ as
	$$\upind_\mu (\oA ) := \sup_{\delta > 0} \upind_\mu (\oA , \delta ).$$
	
	The following useful characterization of $\mu$-UPE  is due to Kerr and Li
	\cite{KerrLi09}.
	
	\begin{theorem}\label{KerrLi} Let $(X,\mu,T)$ be a MDS. Then, $(X,\mu,T)$ has $\mu$-UPE if and only if
		for every pair $\oU=(U_{0},U_{1})$ of nonempty disjoint open sets in $X$, one has $\upind_\mu (\oU )>0$.	\end{theorem}
	

	\section{Our results}
	
	For each $n \in \N$, let
	$$
	\cM_n(X):= \Big\{\displaystyle\frac{1}{n}\sum_{i=1}^n \delta_{x_i}\in\cM(X) :
	x_1,\ldots,x_n \in X \ \text{not necessarily distinct}\Big\},
	$$
	where $\delta_x$ denotes the unit mass concentrated at the point $x$ of $X$.
	It is classical that $\bigcup_{n \in \N} \cM_n(X)$ is dense in $\cM(X)$.
	Since $\cM_n(X)$ is $\wt{T}$-invariant, we can consider the TDS
	$(\cM_n(X),\wt{T})$, where we are also denoting by $\wt{T}$ the
	corresponding restricted map. For each $n\in\mathbb{N}$, let us consider $(X^{(n)},\mu^{(n)})$ the canonical symmetric $n$-fold joining  of $(X,\mu)$ \cite{EGlabook03}, where $\mu^{(n)}:=\mu\times\dots\times\mu$ is the product measure on $X^{(n)}:=X\times\dots\times X$. We also consider $T_n:=T\times\dots\times T$ and, given any $\wt{\mu}\in Q(\mu)$ such that $\widetilde{\mu}(\cM_n(X))>0$ we consider the MDS $(\cM_n(X),\wt{\mu},\wt{T})$, where we are also denoting by $\wt{\mu}$ the corresponding normalized induced measure. Denote by $S_{n}$ the group of all permutations of $n$ elements and let us consider $\tau:X^{(n)}\rightarrow \hat{X}^{(n)}:=X^{(n)}/S_{n}$ the quotient map. A typical element of $\hat{X}^{(n)}$ will be denoted by $\langle x_{1},\dots,x_{n}\rangle$. Moreover, we can consider the quotient measure $\tau_{*}(\mu^{(n)}):=\mu^{(n)}\circ\tau^{-1}$ on $\hat{X}^{(n)}$. Now let us consider the maps $$\psi: ( x_{1},\dots,x_{n})\in X^{(n)} \mapsto (1/n)\displaystyle\sum_{l=1}^{n}\delta_{x_{l}}\in \cM_{n}(X)$$ and $$\hat{\psi}:\langle x_{1},\dots,x_{n}\rangle\in\hat{X}^{(n)}\mapsto(1/n)\displaystyle\sum_{l=1}^{n}\delta_{x_{l}}\in \cM_n(X).$$ Clearly, $\hat{\psi}$ is a Borel isomorphism and we can consider the measure $(\hat{\psi}\circ\tau)_{*}(\mu^{(n)})$ on $\cM_n(X)$. It follows from Lemma 3.5 in \cite{EGlaBWei03} that $(\cM(X),(\hat{\psi}\circ\tau)_{*}(\mu^{(n)}),\tilde{T})$ is a quasifactor of $(X,\mu,T)$ which is isomorphic to $(\hat{X}^{(n)},\hat{\mu}^{(n)},\hat{T}_{n})$, where we are denoting by $\hat{T}_{n}$ the map defined by $\hat{T}_{n}\circ \tau=\tau\circ T_{n}$ and $\hat{\mu}^{(n)}:=\tau_{*}(\mu^{(n)})$. Moreover, it is easy to verify that if $(X^{(n)},\mu^{(n)},T_{n})$ has $\mu^{(n)}$-UPE, then $(\hat{X}^{(n)},\hat{\mu}^{(n)},\hat{T}_{n})$ has $\hat{\mu}^{(n)}$-UPE.
	
	We are ready to state and prove our first result:

	\begin{theorem}\label{First result}
		For every ergodic MDS $(X,\mu,T)$, the following assertions are equivalent:
		\begin{itemize}
			\item [\rm (i)]   $(X,\mu,T)$ has $\mu$-UPE;
			\item [\rm (ii)]  $(\cM_n(X),(\hat{\psi}\circ\tau)_{*}(\mu^{(n)}),\wt{T})$ has $(\hat{\psi}\circ\tau)_{*}(\mu^{(n)})$-UPE for some $1 \leq n <\infty$;
			\item [\rm (iii)] $(\cM_n(X),(\hat{\psi}\circ\tau)_{*}(\mu^{(n)}),\wt{T})$ has $(\hat{\psi}\circ\tau)_{*}(\mu^{(n)})$-UPE for every $1 \leq n < \infty$.
		\end{itemize}                 
	\end{theorem}
	
	\begin{proof}
		
		\smallskip
		\noindent
		(i) $\Rightarrow$ (iii): Suppose that $(X,\mu,T)$ has $\mu$-UPE.
		Fix any $1 \leq n< \infty$ and let $\wt{\oU}^{(n)}:=(\wt{U}^{(n)}_0,\wt{U}^{(n)}_1)$ be any pair of nonempty disjoint open
		sets in $\cM_{n}(X)$. By Lemma~\ref{newbasis}, there exist nonempty open sets
		$U_{0,1},\ldots,U_{0,n},U_{1,1},\ldots,U_{1,n}$ in $X$ such that
		\begin{equation}\label{E1}
			\psi(U_{0,1} \times\cdots\times U_{0,n}) \subset \wt{U}^{(n)}_0
			\ \text{ and } \
			\psi(U_{1,1} \times\cdots\times U_{1,n}) \subset \wt{U}^{(n)}_1
		\end{equation} Put
		\begin{equation*}
			U^{(n)}_0:=U_{0,1} \times\cdots\times U_{0,n}\ \text{and}\	
			U^{(n)}_1:=U_{1,1} \times\cdots\times U_{1,n}\ ,
		\end{equation*}
		which are nonempty open sets in $X^{(n)}$.
		Thus, 
		\begin{equation*}
			\hat{U}^{(n)}_0:=\tau(U^{(n)}_0)\ \text{and}\ \hat{U}^{(n)}_1:=\tau(U^{(n)}_1)
		\end{equation*}
		are nonempty open sets in $\hat{X}^{(n)}$ such that
		\begin{equation}\label{E1.1}
			\hat{\psi}(\hat{U}_0^{(n)})\subset \wt{U}^{(n)}_0\ \text{and}\  \hat{\psi}(\hat{U}_1^{(n)})\subset \wt{U}^{(n)}_1
		\end{equation} Note that $\wt{T} \circ \hat{\psi} = \hat{\psi} \circ \hat{T}_n$.
		Since $\wt{U}^{(n)}_0 \cap \wt{U}^{(n)}_1 = \varnothing$ (\ref{E1.1}) implies that
		\begin{equation}\label{E2}
			\hat{U}^{(n)}_0\cap \hat{U}^{(n)}_1
			= \varnothing.
		\end{equation}
		Now, since any $n$-product of $\mu$-UPE systems is a $\mu^{(n)}$-UPE system \cite{HuangYe} (see also \cite{KerrLi09}), by (\ref{E2}) and Theorem~\ref{KerrLi},
		there exist $d>0$, $\delta>0$ and $m_{k}\to\infty$ such that, for every $k\geqslant 1$ and every $\hat{D}^{(n)}\subset \hat{X}^{(n)}$ with $\hat{\mu}^{(n)}(\hat{D}^{(n)})\geqslant 1-\delta$ there exists $J\subset\mathbb{N}$ with $|J\cap\{1,\dots,m_{k}\}|\geqslant d.m_{k}$ such that $$\hat{D}^{(n)}\cap\bigcap_{j \in I} \hat{T}_{n}^{-j}(\hat{U}^{(n)}_{\sigma(j)}) \neq \varnothing,$$ for every nonempty finite subset 
		$I \subset J$ and for every map $\sigma : I \to \{0,1\}$. Suppose that $k\geqslant 1$ and $\wt{D}^{(n)}\subset \cM_{n}(X)$ is any Borel set such that $(\hat{\psi}\circ\tau)_{*}(\mu^{(n)})(\wt{D}^{(n)})\geqslant 1-\delta$. So, $\hat{D}^{(n)}:=\hat{\psi}^{-1}(\wt{D}^{(n)})$ is such that $\hat{\psi}(\hat{D}^{(n)})=\wt{D}^{(n)}$ and $\hat{\mu}^{(n)}(\hat{D}^{(n)})\geqslant 1-\delta$. Let $J\subset\mathbb{N}$ be as above and pick any nonempty finite set $I\subset J$. Thus, for any map $\sigma:I\to\{0,1\}$, we have
		\begin{align*}
			\varnothing\neq  \hat{\psi}\Big(\hat{D}^{(n)}\cap\bigcap_{j \in I} \hat{T}_n^{-j}(\hat{U}^{(n)}_{\sigma(j)})\Big)\
			&\subset \hat{\psi}(\hat{D}^{(n)})\cap\bigcap_{j \in I} (\hat{\psi}\circ\hat{T}^{-j}_n)\big( \hat{U}^{(n)}_{\sigma(j)}\big)\\
			&\subset\hat{\psi}(\hat{D}^{(n)})\cap\bigcap_{j \in I} \wt{T}^{-j}\big(\hat{\psi}\big( \hat{U}^{(n)}_{\sigma(j)}\big)\big)\\
			&\subset \wt{D}^{(n)}\cap\bigcap_{j \in I} \wt{T}^{-j}(\wt{U}_{\sigma(j)}),		
		\end{align*}
		where we have used (\ref{E1.1}) in the second inclusion. We conclude that $\upind_{(\hat{\psi}\circ\tau)_{*}(\mu^{(n)})} (\wt{\oU}^{(n)})\geqslant d$. Hence, by Theorem~\ref{KerrLi}, we see that $(\cM_n(X),(\hat{\psi}\circ\tau)_{*}(\mu^{(n)}),\wt{T})$ has $(\hat{\psi}\circ\tau)_{*}(\mu^{(n)})$-UPE.
		
		\smallskip
		\noindent
		(ii) $\Rightarrow$ (i): Suppose that $(\cM_n(X),(\hat{\psi}\circ\tau)_{*}(\mu^{(n)}),\wt{T})$ has $(\psi\circ\tau)_{*}(\mu^{(n)})$-UPE for some $2 \leq n < \infty$.
		Let $\oU:=(U_0,U_1)$ be any pair of nonempty disjoint open sets in $X$ and consider
		$$
		\wt{U}^{(n)}_0:= \Big\{\mu \in \cM_n(X) : \mu(U_0) > \frac{n-1}{n}\Big\}
		\ \text{ and } \
		\wt{U}^{(n)}_1:= \Big\{\mu \in \cM_n(X) : \mu(U_1) > \frac{n-1}{n}\Big\},
		$$
		which are nonempty disjoint open sets in $\cM_n(X)$.
		Since $(\cM_n(X),(\hat{\psi}\circ\tau)_{*}(\mu^{(n)}),\wt{T})$ has $(\hat{\psi}\circ\tau)_{*}(\mu^{(n)})$-UPE, the pair $\wt{\oU}:=(\wt{U}^{(n)}_0,\wt{U}^{(n)}_1)$ has positive upper $(\hat{\psi}\circ\tau)_{*}(\mu^{(n)})$-density, that is, there are $d>0$, $\delta>0$ and $m_{k}\to\infty$ such that, for every $k\geqslant 1$ and every $\wt{D}^{(n)}\subset \cM_{n}(X)$ with $(\hat{\psi}\circ\tau)_{*}(\mu^{(n)})(\wt{D}^{(n)})\geqslant 1-\delta$, there exists $J\subset\mathbb{N}$ with $|J\cap\{1,\dots,m_{k}\}|\geqslant d.m_{k}$ such that $$\wt{D}^{(n)}\cap\bigcap_{j \in I} \wt{T}^{-j}(\wt{U}^{(n)}_{\sigma(j)}) \neq \varnothing,$$ for every nonempty finite subset 
		$I \subset J$ and for every map $\sigma : I \to \{0,1\}$. Fix any $0<\alpha<1-\sqrt[n]{1-\delta}$ and pick $k\geqslant 1$ and $D\subset X$ with $\mu(D)\geqslant 1-\alpha$.
		If $D^{(n)}:=D\times\dots\times D$, then $\mu^{(n)}(D^{(n)})\geqslant 1-\delta$ and notice that the set $\wt{D}_{n}:=(\hat{\psi}\circ\tau)(D^{(n)})\subset \cM_{n}(X)$ is so that $(\hat{\psi}\circ\tau)_{*}(\mu^{(n)})(\wt{D}_{n})=\mu^{(n)}(D^{(n)})\geqslant 1-\delta$. So, there exists some $J\subset\mathbb{N}$ with $|J\cap\{1,\dots,m_{k}\}|\geqslant d.m_{k}$ such that $$\wt{D}_{n}\cap\bigcap_{j \in I} \wt{T}^{-j}(\wt{U}^{(n)}_{\sigma(j)}) \neq \varnothing,$$ for every nonempty finite subset 
		$I \subset J$ and for every map $\sigma : I \to \{0,1\}$. Fix any nonempty finite subset 
		$I \subset J$ and any map $\sigma : I \to \{0,1\}$, and pick any $\nu\in\wt{D}_{n}\cap\displaystyle\bigcap_{j \in I} \wt{T}^{-j}(\wt{U}^{(n)}_{\sigma(j)}) \neq \varnothing$. So, there are $x_{1},\dots,x_{n}\in D$ such that $$\nu=\displaystyle\frac{1}{n}\sum_{l=1}^{n}\delta_{x_{l}}$$ and $$\left(\displaystyle\frac{1}{n}\sum_{l=1}^{n}\delta_{T^{j}x_{l}}\right)(U_{\sigma(j)})>\displaystyle\frac{n-1}{n}$$ for every $j\in I$. Therefore, $\{x_{1},\dots,x_{n}\}\subset D\cap\displaystyle\bigcap_{j\in I}T^{-j}(U_{\sigma(j)})$. Since $\mu(D)\geqslant 1-\alpha$ we obtain $\upind_\mu (\oU )\geqslant d$ and so, by Theorem~\ref{KerrLi}, we conclude that $(X,\mu,T)$ has $\mu$-UPE.

	\end{proof}
	In order to prove our next result we will need to recall some notation and also two important results.
	
	For each $n\geqslant 1$, $\delta>0$ and $\mathcal{P}$ a finite measurable partition, let us denote by $N(T,\cP,n,\delta)$ the minimal cardinality of a Borel subcollection of $\cP \lor T^{-1}\cP \lor\dots\lor T^{-(n-1)}\cP$ needed to cover a set $D\subset X$ with $\mu(D)\geqslant 1-\delta$.  
	The following result will be essential to our argument (see \cite{Katok80,Rudolph}):
	
	\begin{theorem}\label{Katok}
		If $(X,\mu,T)$ is ergodic, then $h_{\mu}(T,\cP)=\displaystyle\lim_{n\to\infty}(1/n)\log(N(T,\cP,n,\delta))$, for each fixed $0<\delta<1$.
	\end{theorem}
	
	Now, let us denote by $\ell_1^k$  the vector space $\R^k$ endowed with
	the $\ell_1$-norm, that is,
	$\|(r_1,\ldots,r_k)\|:= |r_1| + \cdots + |r_k|$,
	and $\ell_\infty^m$ the vector space $\R^m$ endowed with the
	$\ell_\infty$-norm, that is,
	$\|(s_1,\ldots,s_m)\|:= \max\{|s_1|,\ldots,|s_m|\}$.
	Moreover, let us denote by $B_{\ell_1^k}$ the closed unit ball of the Banach space
	$\ell_1^k$.
	
	The following quantitative tecnique connecting combinatorics to linear maps on finite-dimensional spaces was developed by Glasner and Weiss (Proposition~2.1 from \cite{EGlaBWei95}):
	
	\begin{lemma}\label{GlasnerWeisstecnique}
		Given constants $\eps > 0$ and $b > 0$, there exist constants $m_0 \in \N$
		and $c > 0$ such that the following property holds for every $m \geq m_0$:
		if $\varphi : \ell_1^k \to \ell_\infty^m$ is a linear map with
		$\|\varphi\| \leq 1$, and if $\varphi(B_{\ell_1^k})$ contains more than
		$2^{bm}$ vectors that are $\eps$-separated, then $k \geq 2^{cm}$.
	\end{lemma}
	
	Let us now establish our next result, which was inspired by its topological counterpart from \cite{NBerUDarRVer22}:
	
	\begin{theorem}\label{Second result}
		For every ergodic MDS $(X,\mu,T)$ and every $\wt{\mu}\in Q(\mu)$, if $(\cM(X),\wt{\mu},\wt{T})$ has $\wt{\mu}$-UPE, then $(X,\mu,T)$ has $\mu$-UPE. 
	\end{theorem}
	
	\begin{proof}
		
		Suppose that $(X,\mu,T)$ is ergodic and that $\wt{\mu}\in Q(\mu)$ is so that $(\cM(X),\wt{\mu},\wt{T})$ has $\wt{\mu}$-UPE. Let $\cP := \{P_0,P_1\}$ be
		a replete partition of $X$. We have to prove that $h_{\mu}(T,\cP) > 0$.
		For this we take $A_0$ and $A_1$ any pair of nonempty open sets in $X$ with
		$$
		A_0 \subset P_0 \backslash \ov{P_1}, \ \ A_1 \subset P_1 \backslash \ov{P_0}
		\ \ \text{ and } \ \ \ov{A_0} \cap \ov{A_1} = \varnothing.
		$$
		Define
		$$
		\wt{A}_0:= \{\mu \in \cM(X) : \mu(A_0) > 0.9\} \ \text{ and } \
		\wt{A}_1:= \{\mu \in \cM(X) : \mu(A_1) > 0.9\},
		$$
		which are nonempty disjoint open sets in $\cM(X)$. By Theorem~\ref{KerrLi}, there are $d>0$, $\delta>0$ and $m_{l}\to\infty$ such that, for every $l\geqslant 1$ and every $\wt{D}\subset \cM(X)$ with $\wt{\mu}(\wt{D})\geqslant 1-\delta$, there exists $J\subset\mathbb{N}$ with $|J\cap\{1,\dots,m_{l}\}|\geqslant d.m_{l}$ such that
		\begin{equation}\label{indcondsecondresult}
			\wt{D}\cap\bigcap_{j \in I} \wt{T}^{-j}(\wt{A}_{\sigma(j)}) \neq \varnothing
		\end{equation}
		for every nonempty finite subset 
		$I \subset J$ and for every map $\sigma : I \to \{0,1\}$. Let $D\subset X$ be any Borel set with $\mu(D)\geqslant 1-\delta^2$ and fix some small $\varepsilon=\varepsilon(\delta)>0$ to be specified later. Let $n_0 \in \N$ and $c > 0$ be constants
		associated to $\eps$ and $d>0$ according to Lemma~\ref{GlasnerWeisstecnique}. Fix $l\geqslant 1$ large enough so that $m:=m_{l} \geq n_0$  and take $\{C_1,\dots,C_{k_{m}}\}$ a subcollection of  $\cP^{m-1}$ with minimal cardinality that covers $D$. Let us consider
		$$
		B_1:= C_1 \ \text{ and } \
		B_i:= C_i \backslash (C_1 \cup \ldots \cup C_{i-1}) \
		\text{ for } 2 \leq i \leq k_m.
		$$
		As $\{C_1,C_2,\ldots,C_{k_m}\}$ is minimal, we have
		$B_i \cap D\neq \varnothing$ for every $i$. Let
		$$
		M:= \big[t_{i,j}\big]_{1 \leq i \leq k_m, 0 \leq j \leq m-1}
		$$
		be a $k_m \times m$ matrix of 0's and 1's such that
		$$
		B_i \subset P_{t_{i,0}} \cap T^{-1}(P_{t_{i,1}}) \cap T^{-2}(P_{t_{i,2}})
		\cap \ldots \cap T^{-(m-1)}(P_{t_{i,m-1}}),
		$$
		for all $1 \leq i \leq k_m$. Consider the linear map
		$\varphi : \ell_1^{k_m} \to \ell_\infty^m$ given by
		$$
		\varphi(r_1,\ldots,r_{k_m}):= [r_1 \ \cdots \ r_{k_m}]\, M.
		$$
		Clearly, $\|\varphi\| \leq 1$. Since $\displaystyle\int_{\mathcal{M}(X)}\nu(D) d\wt{\mu}(\nu)=\mu(D)\geqslant 1-\delta^2$, by the Markov inequality there exists some $\widetilde{D}\subset \mathcal{M}(X)$ with $\widetilde{\mu}(\widetilde{D})\geqslant 1-\delta$ such that $\nu(D)\geqslant 1-\delta$ for each $\nu\in\widetilde{D}$. Now, by (\ref{indcondsecondresult}) there exists some $J\subset\mathbb{N}$ with $|J\cap\{1,\dots,m\}|\geqslant d.m$ such that 
		
		\begin{equation*}\widetilde{D}\cap\bigcap_{j \in I} \wt{T}^{-j}(\wt{A}_{\sigma(j)}) \neq \varnothing
		\end{equation*}
		for every nonempty finite subset 
		$I \subset J$ and for every map $\sigma : I \to \{0,1\}$. Take $I:=J\cap\{1,\dots,m\}$. For any $\sigma:I\to \{0,1\}$, there exists $\nu_{\sigma}\in\widetilde{D}$ such that $\wt{T}^{j}\nu_{\sigma}\in\wt{A}_{\sigma(j)}$ for each $j\in I$. Let $\sigma, \sigma' : I \to \{0,1\}$ be distinct functions and let
		$s \in J$ be such that $\sigma(s) = 1$ and $\sigma'(s) = 0$ (say). Then,
		$$
		\nu_\sigma(T^{-s}(A_1)) > 0.9 \ \ \text{ and } \ \
		\nu_{\sigma'}(T^{-s}(A_0)) > 0.9.
		$$
		Since $T^{-s}(A_1)\cap T^{-s}(P_0)=\varnothing$, we have
		$$
		\bigcup \{B_i : t_{i,s} = 1\} \subset T^{-s}(P_1)\\\text{ and } \ \
		T^{-s}(A_1) \subset \bigcup \{B_i : t_{i,s} = 1\}\cup D^{c}
		$$
		Thus, we obtain
		$$
		\sum_{i=1}^{k_m} t_{i,s}\, \nu_{\sigma'}(B_i)<0.1\ \ \text{ and } \ \
		0.9-\delta<\sum_{i=1}^{k_m} t_{i,s}\, \nu_\sigma(B_i).
		$$                         
		Hence, the $s^\text{th}$ coordinates of the vectors
		$$
		\varphi\big(\nu_\sigma(B_1),\ldots,\nu_\sigma(B_{k_m})\big)
		\ \ \text{ and } \ \
		\varphi\big(\nu_{\sigma'}(B_1),\ldots,\nu_{\sigma'}(B_{k_m})\big)
		$$
		are greater than $0.9-\delta$ and smaller than $0.1$, respectively, showing that, if $\varepsilon=\varepsilon(\delta)>0$ is chosen small enough, then
		these vectors are $\eps$-separated. Since $|I| \geqslant d.m$, there
		are at least $2^{dm}$ functions $\sigma$. This shows that
		$\varphi(B_{\ell_1^{k_m}})$ contains at least $2^{dm}$ vectors that
		are $\eps$-separated. Thus, by Lemma~\ref{GlasnerWeisstecnique}, $k_m \geq 2^{cm}$.
		By Theorem~\ref{Katok}, this implies that $h_{\mu}(T,\cP) \geq c > 0$, as desired.
	\end{proof}
	
	\begin{remark}
		It was proved in \cite{EGlaBWei03} that any ergodic system of positive entropy admits \emph{every} ergodic system of positive entropy as a measure-theoretic quasifactor. As a consequence, for each ergodic $\mu$-UPE system $(X,\mu,T)$ there exists some ergodic quasifactor of it \emph{without} the measure-theoretic UPE property. This shows that, unlike its topological counterpart from \cite{NBerUDarRVer22}, the converse of Theorem~\ref{Second result} does \emph{not} hold.
	\end{remark}



\section*{Acknowledgement}

The author would like to thank Nilson Bernardes Jr. and the anonymous referees for valuable comments that improved the text.


\end{document}